\documentclass[a4paper,11pt]{article}
\usepackage{amsmath,amssymb,amsthm,graphicx,subcaption,xcolor}
\usepackage{fullpage}
\usepackage{enumerate}

\long\def\symbolfootnote[#1]#2{\begingroup%
\def\thefootnote{\fnsymbol{footnote}}\footnote[#1]{#2}\endgroup} 

\newtheorem{theorem}{Theorem}[section]
\newtheorem{lemma}[theorem]{Lemma}

\newtheorem{proposition}[theorem]{Proposition}

\newcommand{\Z}{\mathbb{Z}}
\newcommand{\floor}[1]{{\left\lfloor#1\right\rfloor}}
\newcommand{\ceil}[1]{{\left\lceil#1\right\rceil}}
\newcommand{\oA}{\overline{A}}
\newcommand{\oB}{\overline{B}}

\allowdisplaybreaks[3]

\begin{document}
\title{Additive triples in groups of odd prime order}
\author{Sophie Huczynska \and Jonathan Jedwab \and Laura Johnson}
\date{7 May 2024}
\maketitle

\symbolfootnote[0]{
S.~Huczynska and L.~Johnson are with School of Mathematics and Statistics,
University of St Andrews, Mathematical Institute, North Haugh, St Andrews KY16 9SS, Scotland.
Email: {\tt sh70@st-andrews.ac.uk, lj68@st-andrews.ac.uk}

J.~Jedwab is with Department of Mathematics, 
Simon Fraser University, 8888 University Drive, Burnaby BC V5A 1S6, Canada.
Email: {\tt jed@sfu.ca}

S.~Huczynska was funded by EPSRC grant EP/X021157/1.
J.~Jedwab is supported by NSERC.
}

\begin{abstract}
Let $p$ be an odd prime.
For nontrivial proper subsets $A,B$ of~$\mathbb{Z}_p$ of cardinality $s,t$, respectively, we count the number $r(A,B,B)$ of \emph{additive triples}, namely elements of the form $(a, b, a+b)$ in $A \times B \times B$.
For given $s,t$, what is the spectrum of possible values for $r(A,B,B)$?
In the special case $A=B$, the additive triple is called a \emph{Schur triple}. Various authors have given bounds on the number $r(A,A,A)$ of Schur triples, and shown that the lower and upper bound can each be attained by a set $A$ that is an interval of $s$ consecutive elements of $\mathbb{Z}_p$. However, there are values of $p,s$ for which not every value between the lower and upper bounds is attainable. 
We consider here the general case where $A,B$ can be distinct.
We use Pollard's generalization of the Cauchy-Davenport Theorem to derive bounds on the number $r(A,B,B)$ of additive triples. In contrast to the case $A=B$, we show that every value of $r(A,B,B)$ from the lower bound to the upper bound is attainable: each such value can be attained when $B$ is an interval of $t$ consecutive elements of~$\mathbb{Z}_p$.
\end{abstract}

\section{Introduction}
Let $G$ be an additive group.  
A \emph{Schur triple} in a subset $A$ of $G$ is a triple of the form $(a,b,a+b) \in A^3$; Schur triples were originally considered only in the case $G=\Z$~\cite{Sch}. Let $r(A)$ be the number of Schur triples in~$A$.
Several authors have studied the behaviour of $r(A)$ as $A$ ranges over some or all subsets of a group~$G$, and the nature of the subsets $A$ attaining a particular value of~$r(A)$. 

A \emph{sum-free set} $A$ is one for which $r(A)=0$, and has received much attention.
The Cameron-Erd\H{o}s Conjecture \cite{CamErd} concerns the number of sum-free sets in $\{1,2,\dots,n\} \subset \Z$; this was resolved independently by Green~\cite{Gre} and Sapozhenko~\cite{Sap}.
Lev and Schoen~\cite{LevSch} studied the number of sum-free sets when $G$ is a group of prime order.
Erd\H{o}s \cite{Erd} asked what is the largest size of a sum-free set in an abelian group; this question was considered by Green and Ruzsa~\cite{GreRuz}. 

A popular problem is to determine the minimum and maximum value of $r(A)$ over all subsets $A$ of fixed cardinality in a specified group~$G$. The case $G = \Z_p$ for a prime~$p$ is of particular interest, in part because of its relation to sumset results such as the Cauchy-Davenport Theorem~\cite{Cau,Dav}.
We use the set notation $a+B := \{a+b: b \in B\}$ and $A+B := \{a+B : a \in A\}$.

\begin{theorem}[Cauchy-Davenport Theorem \cite{Cau,Dav}]
\label{thm:c-d}
Let $p$ be prime and let $A,B$ be non-empty subsets of~$\Z_p$.
Then $|A + B| \ge \min(p, |A| + |B|-1)$.
\end{theorem}

The special case $A=B$ of Theorem~\ref{thm:c-d} counts the number of distinct values that the sum $a+b$ can take as $a,b$ range over $A$, without taking account of how many times the sum is attained nor whether it lies in the subset~$A$.

The following generalization of the Cauchy-Davenport Theorem provides more infomation which is relevant to counting occurrences of each sum. The special case $j=1$ reduces to the Cauchy-Davenport Theorem.
\begin{theorem}[Pollard \cite{Pol}]
\label{thm:Pollard}
Let $p$ be prime and let $A, B$ be subsets of $\mathbb{Z}_p$ of cardinality $s,t$, respectively.
For $i \ge 1$, let $S_i$ be the set of elements of $\mathbb{Z}_p$ expressible in at least $i$ ways in the form $a + b$ for $a \in A$ and $b \in B$.
Then
\[
\sum_{i=1}^j |S_i| \ge j \, \min (p,\, s+t-j) \quad \mbox{for $1 \le j \le \min(s,t)$}.
\]
\end{theorem}
\noindent
Theorem~\ref{thm:Pollard} was a crucial tool in the proof of \cite[Theorem~3.6]{HucMulYuc}, which used linear programming to determine the minimum and maximum value of $r(A)$ when $A$ is a subset of fixed cardinality in~$\Z_p$. The following theorem summarizes results from~\cite{HucMulYuc}.

\begin{theorem}[Huczynska, Mullen, Yucas \cite{HucMulYuc}]
\label{thm:hmy}
Let $p$ be an odd prime and let $1 \le s \le p-1$. Let
\begin{align*}
f_s 
 &= \begin{cases}
      0 & \mbox{for $s \le \frac{p+1}{3}$}, \\
      \floor{\frac{(3s-p)^2}{4}} & \mbox{for $\frac{p+2}{3} \le s$}, \\
    \end{cases} \\
 g_s 
 &= \begin{cases}
       \ceil{\frac{3s^2}{4}} & \mbox{for $s \le \frac{2p+1}{3}$},\\
        s(2s-p)+(p-s)^2 & \mbox{for $\frac{2p+2}{3} \le s$}.
    \end{cases} \\
\end{align*}
Then
\begin{enumerate}[$(i)$]
\item
As $A$ ranges over all subsets of $\Z_p$ of cardinality $s$, we have
\[
f_s \le r(A) \le g_s.
\]

\item
The values $f_s$ and $g_s$ for $r(A)$ can each be attained by a set $A$ that is an interval of $s$ consecutive elements of~$\Z_p$.

\item
For certain $p$ and $s$, there is at least one value in the interval $(f_s,\,g_s)$ which is not attainable as $r(A)$ for a subset $A$ of $\Z_p$ of cardinality~$s$.
\end{enumerate}
\end{theorem}

The actual spectrum of possible values of $r(A)$ in the setting of Theorem~\ref{thm:hmy} was conjectured but not resolved in~\cite{HucMulYuc}.
For $p > 11$, not all attainable values of $r(A)$ (found by computer search) were explained by constructions in~\cite{HucMulYuc}.

Samotij and Sudakov \cite{SamSud} obtained similar results to Theorem~\ref{thm:hmy} for various abelian groups, including elementary abelian groups and $\mathbb{Z}_p$, using a different proof to that of~\cite{HucMulYuc}. They also showed that a subset of the group $\mathbb{Z}_p$ achieving the minimum value $f_s$ (when this is nonzero) must be an arithmetic progression.  
Bajnok \cite{Baj} proposed to generalize from counting Schur triples to counting $(s+1)$-tuples, and suggested the case $G=\Z_p$ as a first step. This case was addressed by Chervak, Pikhurko and Staden \cite{ChePikSta}, who showed that extremal configurations exist with all sets consisting of intervals.

In this paper we consider a different generalization of Schur triples. Let $A,B$ be subsets of a group $G$ of cardinality $s,t$, respectively, and let $r(A,B,B)$ be the number of \emph{additive triples} in $G$, namely elements of the form $(a,b,a+b) \in A \times B \times B$. 
(Note that $r(A,A,A)$ is identical to $r(A)$ as used above.)
For given $s,t$, what is the spectrum of possible values of $r(A,B,B)$? This generalization of Schur triples is not only natural, it is also closer to the setting of the Cauchy-Davenport Theorem than is the special case $A=B$.
We shall always take $G=\Z_p$, where $p$ is an odd prime.

Our main result is Theorem~\ref{thm:main}, which determines the smallest and largest value of $r(A,B,B)$ as a function of $s,t$, and shows that (in contrast to the special case $A=B$) every intermediate value can be attained by $r(A,B,B)$.

\begin{theorem}[Main Theorem]
\label{thm:main}
Let $p$ be an odd prime and let $1 \le s,t \le p-1$.  Let
\begin{align}
f(s,t) &=  \begin{cases}
        0 & \mbox{for $2t \le p-s+1$}, \\
        \floor{\frac{(s+2t-p)^2}{4}} & \mbox{for $p-s+2 \le 2t \le p+s-2$}, \\
        s(2t-p) &  \mbox{for $p+s-1 \le 2t$},
           \end{cases} \label{eqn:f} \\
g(s,t) &=  \begin{cases}
        t^2 & \mbox{for $2t \le s$},\\
        \ceil{\frac{s(4t-s)}{4}} & \mbox{for $s+1 \le 2t \le 2p-s-1$},\\
        s(2t-p)+(p-t)^2 & \mbox{for $2p-s \le 2t$}.
           \end{cases} \label{eqn:g}
\end{align}
The set of values taken by $r(A,B,B)$ as $A,B$ range over all subsets of $\Z_p$ of cardinality $s,t$, respectively, is the closed integer interval~$[f(s,t), \, g(s,t)]$.
\end{theorem}

In Section~\ref{sec:bounds} we shall show (for an odd prime $p$) that 
$f(s,t) \le r(A,B,B) \le g(s,t)$ for all subsets $A,B$ of $\Z_p$ of cardinality $s,t$, respectively.
In Section~\ref{sec:constructions} we shall show (for an odd although not necessarily prime~$p$) that for each integer $r \in [f(s,t),\, g(s,t)]$ and for $B=\{0,1,\dots,t-1\}$, there is a subset $A$ of $\Z_p$ of cardinality $s$ for which $r(A,B,B) = r$.
Combining these results proves Theorem~\ref{thm:main}.

It is interesting to note that, while the relaxation from Schur triples to additive triples yields a spectrum of values of $r(A,B,B)$ which no longer has any ``missing values" between the minimum and maximum, the actual values of the minimum and maximum for $r(A,B,B)$ with $|A|=|B|=s$ are precisely the same as the minimum and maximum of $r(A,A,A)$ with $|A|=s$.
Indeed, we see from \eqref{eqn:f} that 
\begin{align*}
f(s,s) 
 &= \begin{cases}
       0 & \mbox{for $s \le \frac{p+1}{3}$}, \\
       \floor{\frac{(3s-p)^2}{4}} & \mbox{for $\frac{p+2}{3} \le s \le p-2$}, \\
       s(2s-p) &  \mbox{for $s = p-1$}
    \end{cases} \\
 &= f_s
\end{align*}
by combining the domain $s=p-1$ with the domain $\frac{p+2}{3} \le s \le p-2$.
We also see from \eqref{eqn:g} that 
\begin{align*}
g(s,s) 
 &= \begin{cases}
       \ceil{\frac{3s^2}{4}} & \mbox{for $s \le \frac{2p-1}{3}$},\\
       s(2s-p)+(p-s)^2 & \mbox{for $\frac{2p}{3} \le s$}
    \end{cases} \\
 &= g_s
\end{align*}
by transferring the cases where $s = \frac{2p}{3}$ or $s = \frac{2p+1}{3}$ is an integer from the domain $\frac{2p}{3} \le s$ to the domain $s \le \frac{2p-1}{3}$.

\section{Preliminary results}
In this section we obtain some preliminary results for additive triples in a group~$G$ (not necessarily~$\Z_p$).
We firstly derive two expressions for $r(A,B,B)$.

\begin{proposition}
\label{prop:twoexp}
Let $G$ be a group and let $A,B$ be subsets of~$G$.
\begin{enumerate}[$(i)$]
\item
We have
\[
r(A,B,B) = \sum_{a \in A} \big |(a+B) \cap B \big|. 
\]

\item
For each $i \ge 1$, let $S_i$ be the set of elements of $G$ expressible in at least $i$ ways in the form $a+b$ for $a \in A$ and $b \in B$. Then
\[
    r(A,B,B) =\sum_{i \ge 1} |S_i \cap B|.
\]
\end{enumerate}
\end{proposition}

\begin{proof}
\mbox{}
\begin{enumerate}[$(i)$]
\item
By definition,
\begin{align*}
r(A,B,B) 
 &= \big | \{ (a,b,a+b) : a \in A,\, b \in B,\, a+b \in B\} \big | \nonumber \\
 &= \sum_{a \in A} \big | \{ b: b \in B,\, a+b \in B\} \big | \nonumber \\
 &= \sum_{a \in A} \big |(a+B) \cap B \big|.
\end{align*}

\item
Fix $c \in B$ and consider the set $X(c)$ of triples of the form $(a,b,a+b) \in A \times B \times B$ for which $a+b=c$. 
We prove the required equality by showing that the triples of $X(c)$ contribute equally to the left hand side and the right hand side. 
The contribution to the left hand side is~$|X(c)|$. 
The contribution to $|S_i \cap B|$ is 1 for each $i$ satisfying $1 \le i \le |X(c)|$ and is 0 for each $i > |X(c)|$, giving a total contribution to the right hand side of~$|X(c)|$.
\end{enumerate}

\end{proof}

Write $\oA$ for the complement of a subset $A$ in a group~$G$.
We now give a relationship between $r(A,B,B)$ and $r(\oA,\oB,\oB)$.

\begin{theorem}\label{thm:relation}
Let $A,B$ be subsets of a group~$G$. Then
\[
r(A,B,B)+ r(\oA,\oB,\oB)= |A| \cdot |B| - |A| \cdot |\oB| + |\oB|^2.
\]
\end{theorem}
\begin{proof}
We calculate
\begin{align*}
 & r(A,B,B)+ r(\oA,\oB,\oB) \\
 & \hspace{10mm} = \Big( r(A,B,B) + r(A,B,\oB) \Big)
  - \Big( r(A,B,\oB) + r(A,\oB,\oB) \Big) 
  + \Big( r(A,\oB,\oB) + r(\oA,\oB,\oB) \Big) \\
 & \hspace{10mm} = |A| \cdot |B| - |A| \cdot |\oB| + |\oB|^2
\end{align*}
by definition of $r(A,B,B)$.

\end{proof}

\section{Establishing the lower and upper bounds}
\label{sec:bounds}
In this section we prove Theorem~\ref{thm:fg} below, which establishes a lower and upper bound on the value of $r(A,B,B)$ for all subsets $A$ and~$B$.

\begin{theorem}
\label{thm:fg}
Let $p$ be an odd prime, let $1 \le s,t \le p-1$, and let $A,B$ be subsets of $\Z_p$ of cardinality $s,t$, respectively.
Let $f(s,t)$ and $g(s,t)$ be the functions defined in \eqref{eqn:f} and~\eqref{eqn:g}. Then $f(s,t) \le r(A,B,B) \le g(s,t)$.
\end{theorem}
\begin{proof}
We make the following claim, which will be proved subsequently:
\begin{equation}
\label{eqn:lowerbound}
r(X,Y,Y) \ge f(|X|,|Y|) \quad \mbox{for all subsets $X,Y$ of $\Z_p$}.
\end{equation}
Given this claim, by Theorem~\ref{thm:relation} we have
\begin{align}
r(A,B,B) 
 &= st-s(p-t)+(p-t)^2-r(\oA,\oB,\oB) \nonumber \\
 &\le st-s(p-t)+(p-t)^2-f(p-s,p-t) \label{eqn:upperbound}
\end{align}
using the case $(X,Y) = (\oA,\oB)$ of~\eqref{eqn:lowerbound}. 
By definition of $f$, we have
\begin{align*}
f(p-s,p-t)
 & = \begin{cases}
       (p-s)(p-2t) &  \mbox{for $2t \le s+1$}, \\
       \floor{\frac{(2p-s-2t)^2}{4}} & \mbox{for $s+2 \le 2t \le 2p-s-2$}, \\
       0 & \mbox{for $2p-s-1 \le 2t$},
     \end{cases}  
\end{align*}
and we may adjust the three ranges for $2t$ to give the equivalent form
\[
f(p-s,p-t)
 = \begin{cases}
        (p-s)(p-2t) &  \mbox{for $2t \le s$}, \\
        \floor{\frac{(2p-s-2t)^2}{4}} & \mbox{for $s+1 \le 2t \le 2p-s-1$}, \\
        0 & \mbox{for $2p-s \le 2t$}.
     \end{cases}  
\]
Substitution in \eqref{eqn:upperbound} and straightforward calculation then gives 
\[
r(A,B,B) \le g(s,t),
\]
which combines with the case $(X,Y) = (A,B)$ of \eqref{eqn:lowerbound} to give the required result.

It remains to prove the claim~\eqref{eqn:lowerbound} by showing that
$r(A,B,B) \ge f(s,t)$. 
Our argument is inspired by that used in the proof of~\cite[Theorem 1.3]{SamSud}.
For $i \ge 1$, let $S_i$ be the set of elements of $\Z_p$ expressible in at least $i$ ways in the form $a+b$ for $a \in A$ and $b \in B$.  
By Proposition~\ref{prop:twoexp}$(ii)$, for $j \ge 1$ we have
\begin{eqnarray*}
r(A,B,B) 
    & \ge & \sum_{i=1}^j |S_i \cap B|\\
    & \ge & \sum_{i=1}^j \big(|S_i|- |\oB|\big)
\end{eqnarray*}
using the set inequality $|S_i \cap B|+ | \oB| \ge |S_i|$.
Theorem~\ref{thm:Pollard} then gives 
\begin{equation}
\label{eqn:pollardr}
r(A,B,B) \ge j\, \min(p, \, s+t-j)- j(p-t)  \quad \mbox{for $1 \le j \le \min(s,t)$}.
\end{equation}

\begin{description}
\item[Case 1: $2t \le p-s+1$.] 
In this range, $r(A,B,B) \ge 0$ trivially.

\item[Case 2: $p-s+2 \le 2t \le p+s-2$.] 
In this range, the value $j=\ceil{\frac{s+2t-p}{2}}$ satisfies $1 \le j < \min(s,t)$ and $s+t-j < p$, 
so substitution in \eqref{eqn:pollardr} gives
\begin{align*}
r(A,B,B) 
  &\ge j(s+t-j)- j(p-t)  \\
  &= j(s+2t-p-j)  \\
  &= \floor{\frac{(s+2t-p)^2}{4}}.
\end{align*}

\item[Case 3: $p+s-1 \le 2t$.]
In this range, the value $j=s$ satisfies $1 \le j \le \min(s,t)$ and $s+t-j < p$, 
so substitution in \eqref{eqn:pollardr} gives
\begin{align*}
r(A,B,B) 
  &\ge j(s+t-j)- j(p-t)  \\
  &= s(2t-p).
\end{align*}

\end{description}
Combining results for Cases 1, 2, and 3 proves that $r(A,B,B) \ge f(s,t)$, as required.
\end{proof}

\section{Achieving the spectrum constructively}
\label{sec:constructions}
In this section we constructively prove Theorem~\ref{thm:construction} below, which shows that each integer value~$r$ in the closed interval $[f(s,t),\, g(s,t)]$ is an attainable value of $r(A,B,B)$ for some choice of subsets $A$ and~$B$. The construction takes $p$ to be odd but does not require $p$ to be prime.

\begin{theorem} 
\label{thm:construction}
Let $p$ be an odd integer, let $1 \le s,t \le p-1$, and let $B = \{0,1,\dots,t-1\}$. 
Let $f(s,t)$ and $g(s,t)$ be the functions defined in \eqref{eqn:f} and \eqref{eqn:g}, and let $r \in [f(s,t),\, g(s,t)]$. Then there is a subset $A$ of $\Z_p$ of cardinality $s$ for which $r(A,B,B) = r$.
\end{theorem}

We shall use a visual representation of a multiset involving balls and urns. For example, Figure~\ref{fig:M}(a) represents the multiset comprising $p-2t+1$ elements~$0$, two elements each of $1,2,\dots,t-1$, and one element~$t$.
We firstly use Proposition~\ref{prop:twoexp}$(i)$ to transform the condition $r(A,B,B) = r$ into an equivalent statement involving the multiset in Figure~\ref{fig:M}.

\begin{figure}[ht]
  \centering
  \begin{subfigure}[b]{0.8\linewidth}
    \includegraphics[width=\linewidth]{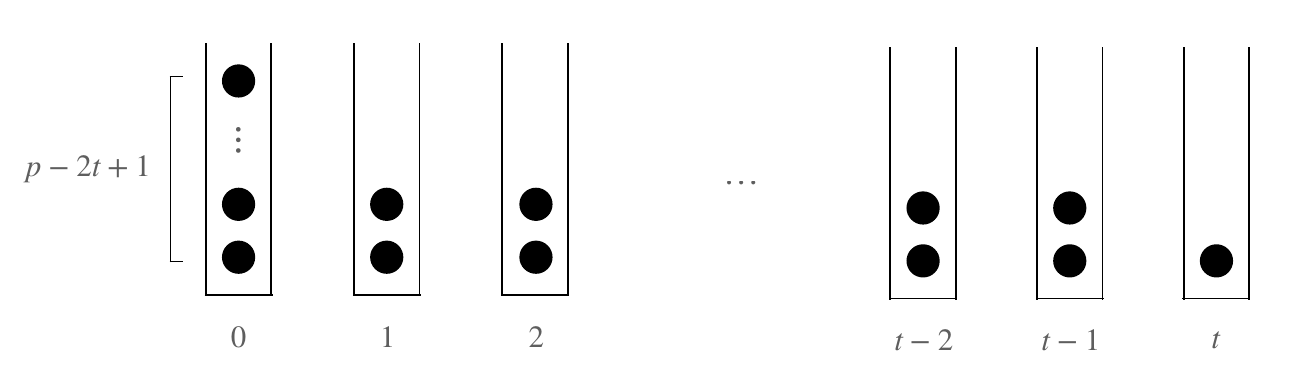}
    \caption{The case $2t \le p-1$}
    \vspace{1.5em}
  \end{subfigure}
 
  \begin{subfigure}[b]{0.8\linewidth}
    \includegraphics[width=\linewidth]{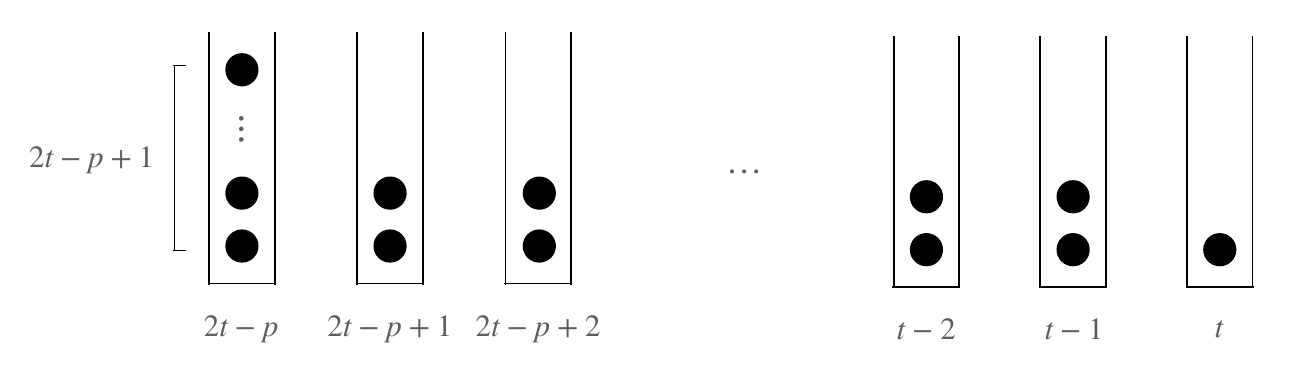}
    \caption{The case $2t \ge p+1$}
  \end{subfigure}
  \caption{The multiset $M$, according to whether $2t \le p-1$ or $2t \ge p+1$.}
  \label{fig:M}
\end{figure}

\begin{lemma}
\label{lem:multiset}
Let $p$ be an odd integer, let $s,t$ be integers satisfying $1 \le s,t \le p-1$, and let
$B = \{0,1,\dots,t-1\}$. 
Then there is a subset $A$ of $\Z_p$ of cardinality $s$ for which $r(A,B,B) = r$ if and only if the multiset $M$ represented in Figure~\ref{fig:M} contains a multi-subset of cardinality $s$ whose elements sum to~$r$.
\end{lemma}

\begin{proof}
Regard $\Z_p$ as having representatives
$\big\{0,\pm 1,\pm 2,\dots,\pm(\frac{p-1}{2})\big\}$, and let $A$ be a subset of~$\Z_p$.
We make the following claim, which will be proved subsequently: for $a \in \{0,1,\dots,\frac{p-1}{2}\}$,
\begin{equation}
\label{eqn:a-a}
|(a+B) \cap B| = |(-a+B) \cap B| 
  = \begin{cases}
  \max(0, t-a)      & \mbox{for $2t \le p-1$}, \\  
  \max(t-a, 2t-p)   & \mbox{for $2t \ge p+1$}
  \end{cases}.
\end{equation}
Given this claim, as $a$ ranges over $\Z_p =  
\big\{0,\pm 1,\pm 2,\dots,\pm(\frac{p-1}{2})\big\}$,
the size of the intersection $|(a+B) \cap B|$ takes each value in the multiset~$M$ (having cardinality~$p$) exactly once.
It then follows from Proposition~\ref{prop:twoexp}$(i)$ that there is a subset $A$ of $\Z_p$ of cardinality $s$ for which $r(A,B,B) = r$ if and only if $M$ contains a multi-subset of cardinality $s$ whose elements sum to~$r$.

It remains to prove the claim. 
Let $a \in \{0,1,\dots,\frac{p-1}{2}\}$.
It is sufficient to prove that $|(a+B) \cap B|$ takes the form stated in \eqref{eqn:a-a}, because $|(-a+B) \cap B| = |(a+(-a+B)) \cap (a+B)| = |B \cap (a+B)|$.

\begin{description}
\item[Case 1: $2t \le p-1$.] 
Since $a+t-1 \le \frac{p-1}{2}+ \frac{p-1}{2} -1 < p$, we have $a+B = \{a,a+1,\dots,a+t-1\}$ (in which reduction modulo $p$ is not necessary) and so
\[
|(a+B) \cap B| = | \{a,a+1,\dots,t-1\} | = \max(0,t-a), 
\]
as required.

\item[Case 2: $2t \ge p+1$.] 
We have
\[
a+B = \begin{cases}
  \{a,a+1,\dots,a+t-1\} & \mbox{for $a+t-1 \le p-1$,} \\
  \{a, a+1, \dots, p-1\} \cup \{0, 1, \dots, a+t-1-p\} & \mbox{for $a+t-1 \ge p$},
\end{cases}
\] and so
\begin{align*}
|(a+B) \cap B| 
 &= \begin{cases}
 t-a  & \mbox{for $a+t-1 \le p-1$,} \\
 (t-a) + (a+t-p)  & \mbox{for $a+t-1 \ge p$} 
 \end{cases} \\
 &= \max(t-a, 2t-p),
\end{align*}
as required.
\end{description}
Combining results for Cases 1 and 2 proves the claim.
\end{proof}

The following counting result is straightforward to verify.

\begin{lemma}
\label{lem:sum-identities}
Let $n, u$ be integers, where $1 \le n \le 2u-1$. Let $S$ be the multiset
\[
\{1,1,2,2,\dots,u-1,u-1\} \cup \{u\}.
\]
Then the sum of the $n$ smallest elements of $S$ is $\floor{\frac{(n+1)^2}{4}}$
and the sum of the $n$ largest elements of $S$ is
$\ceil{\frac{n(4u-n)}{4}}$.
\end{lemma}

We now have the necessary ingredients to prove Theorem~\ref{thm:construction}.

\begin{proof}[Proof of Theorem~\ref{thm:construction}]
We consider the odd integer $p$ and the integers $s,t$ satisfying $1 \le s, t \le p-1$ to be fixed.
Let $M$ be the multiset represented in Figure~\ref{fig:M}, in which we distinguish the cases $2t \le p-1$ and $2t \ge p+1$.
We make the following claim, which will be proved subsequently: the sum $r_1$ of the $s$ smallest elements of $M$ and the sum $r_2$ of the $s$ largest elements of $M$ are given in the following table.

\[ 
    \begin{array}{c|c|c}
    & 2t \le p-1 & 2t \ge p+1 \\[1ex] \hline
    
r_1 &   \begin{cases}
        0  & \mbox{for $s \le p-2t+1$}, \\
        \floor{\frac{(s+2t-p)^2}{4}} & \mbox{for $p-2t+2 \le s$}
        \end{cases}
    
    &   \begin{cases}
        s(2t-p)  & \mbox{for $s \le 2t-p+1$}, \\
        \floor{\frac{(s+2t-p)^2}{4}} & \mbox{for $2t-p+2 \le s$}
        \end{cases} \\[3ex] \hline
    &   &           \\[-2ex]
r_2 &   \begin{cases}
        \ceil{\frac{s(4t-s)}{4}} & \mbox{for $s \le 2t-1$}, \\
        t^2 & \mbox{for $2t \le s$}
        \end{cases}

    &   \begin{cases}
        \ceil{\frac{s(4t-s)}{4}} & \mbox{for $s \le 2p-2t-1$}, \\
        s(2t-p)+(p-t)^2 & \mbox{for $2p-2t \le s$}
        \end{cases}
    \end{array} 
\]
Given this claim, it then follows that for each integer $r \in [r_1, r_2]$
there is a multi-subset of $M$ of cardinality $s$ whose elements sum to $r$:
transform the multi-subset whose elements sum to $r_1$ into the multi-subset whose elements sum to $r_2$ by repeatedly moving some ball one urn to the right until the correct number of balls is contained in urn $t$, then in urn $t-1$, and so on.
By Lemma~\ref{lem:multiset}, for each integer $r \in [r_1,\,r_2]$ and for $B = \{0,1,\dots,t-1\}$ there is therefore a subset $A$ of $\Z_p$ of cardinality $s$ for which $r(A,B,B) = r$.
The ranges for $s,t$ in the above table can be rewritten to emphasize the value of~$2t$ rather than~$s$, and the intervals $[r_1, r_2]$ for the cases $2t \le p-1$ and $2t \ge p+1$ then combined to give the interval $[f(s,t),\, g(s,t)]$ described in Theorem~\ref{thm:construction}.

It remains to prove the claim.
\begin{description}
    \item[Case 1:] $2t \le p-1$. See Figure~\ref{fig:M}(a).

    \begin{description}
        \item[The sum $r_1$.]
        If $s \le p-2t+1$ then the $s$ smallest elements of $M$ are each $0$, so $r_1 = 0$.
        
        Otherwise the sum of the $s$ smallest elements of $M$ is the sum of the first $s-(p-2t+1)$ elements of the multiset $\{1,1,2,2,\dots\,t-1,t-1\} \cup \{t\}$, which by Lemma~\ref{lem:sum-identities} (with $u=t$ and $n = s-(p-2t+1)$) equals $\floor{\frac{(s+2t-p)^2}{4}}$.

        \item[The sum $r_2$.]
        If $s \le 2t-1$ then the sum of the $s$ largest elements of $M$ is the sum of the $s$ largest elements of the multiset $\{1,1,2,2,\dots\,t-1,t-1\} \cup \{t\}$, which by Lemma~\ref{lem:sum-identities} (with $u=t$ and $n = s$) equals
        $\ceil{\frac{s(4t-s)}{4}}$.

        Otherwise the sum of the $s$ largest elements of $M$ is the sum of all elements of the multiset $\{1,1,2,2,\dots,t-1,t-1\} \cup \{t\}$, which equals~$t^2$.
    \end{description}

    \item[Case 2:] $2t \ge p+1$. See Figure~\ref{fig:M}(b).

    \begin{description}
        \item[The sum $r_1$.]

        If $s \le 2t-p+1$ then the $s$ smallest elements of $M$ are each $2t-p$, so $r_1 = s(2t-p)$.
        
        Otherwise the sum of the $s$ smallest elements of $M$ is $s(2t-p)$ plus the sum of the first $s-(2t-p+1)$ elements of the multiset $\{1,1,2,2,\dots,p-t-1,p-t-1\} \cup \{p-t\}$, which by Lemma~\ref{lem:sum-identities} (with $u = p-t$ and $n = s-(2t-p+1)$) equals 
        $s(2t-p) + \floor{\frac{(s-2t+p)^2}{4}} = \floor{\frac{(s+2t-p)^2}{4}}$.

        \item[The sum $r_2$.]
        If $s \le 2p-2t-1$ then the sum of the $s$ largest elements of $M$ is the sum of the $s$ largest elements of the multiset $\{1,1,2,2,\dots,t-1,t-1\} \cup \{t\}$, which by Lemma~\ref{lem:sum-identities} (with $u = t$ and $n = s$) equals $\ceil{\frac{s(4t-s)}{4}}$.

        Otherwise the sum of the $s$ largest elements of $M$ is 
        $s(2t-p)$ plus the sum of all elements of the multiset $\{1,1,2,2,\dots,p-t-1,p-t-1\} \cup \{p-t\}$, which equals $s(2t-p) + (p-t)^2$.
   \end{description}

\end{description}
Combining results for Cases 1 and 2 proves the claim.
\end{proof}

\section{Open questions}
Theorem~\ref{thm:main} gives complete information about the spectrum of $r(A,B,B)$ for subsets $A,B$ of $\Z_p$ of cardinality $s,t$, respectively, for an odd prime~$p$. 

What happens when $p$ is not prime? For example, for $p=9$ the interval $[f(7,6),\,g(7,6)]$ specified by \eqref{eqn:f} and \eqref{eqn:g} is $[25,30]$, but the actual set of attainable values of $r(A,B,B)$ is the larger set $\{24\} \cup [25,30]$. In this example, the value $r(A,B,B) = 24$ is achieved by $A=\{0, 1, 2, 4, 5, 7, 8\}$ and $B=\{0, 1, 3, 4, 6, 7\}$;
the two-way implication of Lemma~\ref{lem:multiset} tells us that this value cannot be achieved by taking $B$ to be the interval $\{0,1,2,3,4,5\}$.

More generally, what can be said about $r(A,B,B)$ when $G$ is not a cyclic group?

\bibliographystyle{abbrv}

\end{document}